\documentclass[12pt]{amsart}%{amsart}

\usepackage{amssymb}
\usepackage[colorlinks,
    linkcolor={red!50!black},
    citecolor={blue!50!black},
    urlcolor={blue!80!black}]{hyperref}
\usepackage{graphicx}
\usepackage{tikz-cd}
\newtheorem{theorem}{Theorem}%[section]
\newtheorem{lemma}[theorem]{Lemma}
\newtheorem{proposition}[theorem]{Proposition}
\newtheorem{corollary}[theorem]{Corollary}

\theoremstyle{definition}
\newtheorem{definition}[theorem]{Definition}

\theoremstyle{remark}

\newcommand{\TryPackage}[3]{\IfFileExists{#1.sty}{\usepackage{#1}#2}{#3}}
\TryPackage{mathrsfs}{\renewcommand{\mathcal}{\mathscr}}{%
     \TryPackage{eucal}{}{}}

\newcommand{\lto}{\longrightarrow}
\newcommand{\wh}{\widehat}
\newcommand{\wt}{\widetilde}
\newcommand{\al}{\alpha}

\newcommand{\ga}{\gamma}

\newcommand{\varep}{\varepsilon}
\renewcommand{\rho}{\varrho}

\newcommand{\la}{\lambda}

\newcommand{\si}{\sigma}
\newcommand{\Ga}{\Gamma}

\newcommand{\La}{\Lambda}
\newcommand{\Si}{\Sigma}

\newcommand{\ZZ}{{\mathbb Z}}
\newcommand{\RR}{{\mathbb R}}
\newcommand{\CC}{{\mathbb C}}

\newcommand{\PP}{{\mathbb P}}

\newcommand{\cI}{{\mathcal I}}

\newcommand{\cR}{{\mathcal R}}

\newcommand{\fX}{{\mathfrak X}}

\newcommand{\mer}{{\mathcal M}} % Meridian
\newcommand{\lng}{{\mathcal L}}  % Longitude
\newcommand{\bmu}{{\mu_J}}
\newcommand{{\bla}}{{\lambda_J}}
\newcommand{\SLC}{{SL(2, {\mathbb C})}}

\newcommand{\tr}{\operatorname{\it tr}}

\newcommand{\Spec}{\operatorname{Spec}}

\newcommand{\sm}{{\smallsetminus}}
\newcommand{\del}{{\partial}}

\begin{document}

% \title[short text for running head]{full title}
\title[The SL$(2,\CC)$ Casson invariant for knots and the $\wh{\rm A}$-polynomial]
{The SL(2,\,$\CC$) Casson invariant for knots \\ and the $\wh{\mathbf A}$-polynomial}

%    Only \author and \address are required; other information is
%    optional.  Remove any unused author tags.

%    author one information
% \author[short version for running head]{name for top of paper}
\author{Hans U. Boden}
\address{Mathematics \& Statistics, McMaster University, Hamilton, Ontario} %, L8S 4K1 Canada}
\curraddr{}
\email{boden@mcmaster.ca}
%\thanks{The first author was supported by a grant from the Natural Sciences and Engineering Research Council of Canada.}

%    author two information
\author{Cynthia L. Curtis}
\address{Mathematics \& Statistics, The College of New Jersey, Ewing, NJ} %, 08628 USA}
\curraddr{}
\email{ccurtis@tcnj.edu}
\thanks{}

%    \subjclass is required.

%\date{\ltoday}

\dedicatory{}

%    "Communicated by" -- provide editor's name; required.
%\commby{}

%    Abstract is required.
\date{\today}
\begin{abstract}
In this paper, we extend the definition of the $\SLC$ Casson invariant  
to arbitrary knots $K$ in integral homology 3-spheres and relate it to the $m$-degree of  the  $\wh{A}$-polynomial of $K$. We prove a product formula for the  $\wh{A}$-polynomial of the connected sum $K_1 \# K_2$ of two knots in $S^3$ and deduce additivity of $\SLC$ Casson knot invariant  under connected sum for a large class of knots in $S^3$. We also present an example of a nontrivial knot $K$ in $S^3$ with trivial $\wh{A}$-polynomial and trivial $\SLC$ Casson knot invariant, showing that neither of these invariants detect the unknot.

\end{abstract}

\maketitle

%    Text of article.

%%%%%%%%%%%%%%%%%%%%%%%%%%%%%%%%%%%%%%%%
\section*{Introduction}

Given a knot $K \subset \Si$ in an integral homology 3-sphere,  let $M = \Si \sm \tau(K)$ denote the complement of $K$ and let $M_{p/q}$ be the result of $p/q$%\tfrac{p}{q})$
-Dehn surgery on $K$. In the case $K$ is a small knot, Theorem 4.8 of \cite{C01} gives a surgery formula for $\la_\SLC (M_{p/q})$, and it follows that the difference
$ \la_\SLC (M_{p/(q+1)}) - \la_\SLC (M_{p/q})$ is independent of $p, q$ provided $p$ and $q$ are relatively prime and 
$q$ is chosen sufficiently large. The $\SLC$ Casson knot invariant is therefore defined for small
knots by setting, for $q \gg 1$,
\begin{equation}\label{smallknotdefn}
\la'_\SLC(K) = \la_\SLC (M_{1/(q+1)}) - \la_\SLC (M_{1/q}).
\end{equation}
In this paper, we present a method for defining the invariant $\la'_\SLC(K)$ more generally for knots in integral homology 3-spheres. 
Unfortunately,  the  surgery formula does not hold for non-small knots; the proof breaks down when $M$ contains a closed essential surface. 
Here, we adopt a different approach and study the asymptotic behavior of $\la_\SLC (M_{p/q})$ as $q \to \infty$, where the limit is taken over all $q$ relatively prime to $p$. As a function in $q$, we prove that $\la_\SLC (M_{p/q})$ has linear growth, and we define $\la'_\SLC(K)$ to be the leading coefficient of $\la_\SLC (M_{p/q})$ as $q \to \infty$.

For small knots, there is a close relationship between the knot invariant  $\la'_\SLC(K)$ and the $m$-degree of the $\wh{A}$-polynomial of $K$,
 which is the $A$-polynomial with multiplicities as defined by Boyer and Zhang in \cite{BZ01}. For instance, in the case of a two-bridge knot $K$, it is known that $\la'_\SLC(K) = \frac{1}{2}\deg_m \wh{A}_K(m,\ell)$, see \cite[Section 3.3]{BC12}.
We extend this relationship to the general setting of knots in homology 3-spheres. We show that the $\wh{A}$-polynomial is multiplicative under connected sums in $S^3$ and deduce additivity of $\la'_\SLC(K)$ under connected sums for most (conjecturally all) knots in $S^3$.  

We conclude the paper with an example of a nontrivial knot $K$ in $S^3$ 
 for which the $\wh{A}$-polynomial  and $\la'_\SLC(K)$ are trivial.
  Thurston classified knots into three types: torus, hyperbolic, and satellite, and for torus and hyperbolic knots, one can show directly that the $m$-degree of the $\wh{A}$-polynomial is nontrivial. Using the relationship between $\la'_\SLC(K)$ and
 the $m$-degree of the $\wh{A}$-polynomial, it then follows that any knot  with $\la'_\SLC(K)=0$ is necessarily a satellite knot.

We therefore consider satellite knots $K$ given as Whitehead doubles, and we examine the $\SLC$ character variety $X(M)$ of the complement $M=S^3 \sm \tau(K)$. We show that for many untwisted doubles,  apart from the component of reducibles, every other component $X_j$ of $X(M)$ has dimension $\dim X_j >1.$  This implies that  $\wh{A}_K(m,\ell) =0$, 
and it shows that the $\wh{A}$-polynomial does not detect the unknot, answering a question raised in \cite{B12}.
Using the relationship between the knot invariant  $\la'_\SLC(K)$ and the $m$-degree of the $\wh{A}$-polynomial of $K$, this implies further that $\la'_\SLC(K)=0$ and
answers the question raised in \cite{BC12} as to whether  the $\SLC$ Casson knot invariant detects the unknot (cf. Theorem 3.3 of \cite{BC12}).

%%%%%%%%%%%%%%%%%%%%%%%%%%%%%%%%%%%%%%%%
\section{Preliminaries}
In this section, we begin by introducing notation for the $\SLC$ representation spaces and character varieties. We also review the definition of the $\SLC$ Casson invariant and surgery formula from \cite{C01}, as well as the $A$-polynomial of \cite{CCGLS} and the $\wh{A}$-polynomial of \cite{BZ01}.

\subsection{Representations and the character variety} \label{rep-char}
Given a finitely generated group $G$, we set $R(G)$ to be the
space of representations $\rho\colon G \lto \SLC$ and
$R^*(G)$ the subspace of irreducible representations. Recall
from \cite{CS83} that $R(G)$ has the structure of a complex affine
algebraic set. The {\sl character}  of a representation $\rho$ is
the function $\chi_\rho\colon G \lto \CC$ defined by setting
$\chi_\rho(g)=\tr(\rho(g))$ for $g \in G$. The set
of characters of $\SLC$ representations 
 admits the structure of a complex affine algebraic set.
 We denote by $X(G)$ the underlying variety of this algebraic set,
and by  $X^*(G)$ the  variety of characters of
irreducible representations.
Define $t\colon R(G) \lto
X(G)$ by $\rho \mapsto \chi_\rho$, and note that $t$ is
surjective.  

Next, we will define the character scheme ${\fX}(G)$ in terms of the universal character ring. Since $G$ is finitely generated, there exist
elements $g_1,\ldots, g_n \in G$ such that, for any $g \in G$, we have a polynomial $P_{G,g} \in \CC[x_1, \ldots, x_n]$ with the property that $\chi_\rho(g)=P_{G,g}(x_1,\ldots, x_n)$
under the substitutions $x_i =\tr \rho(g_i)$ for all $\rho \colon \Ga \to \SLC$.
This assertion follows easily from the Cayley-Hamilton theorem if $G$ is the free group $F_k$ of rank $k$, and in general, 
using a presentation for $G$ to write it as the quotient of $F_k$, we define
$\cR(G)=\CC[x_1, \ldots, x_n]/\cI(G)$, where $\cI(G)=\{P_{F_k,g} \mid g \in \ker(F_k \to G) \}$. The ring $\cR(G)$ can be shown to be independent of the choice of presentation of $G$ and is called the \emph{universal character ring}. The character scheme is defined as ${\fX}(G)=\Spec \cR(G)$ and is said to be \emph{reduced} if $\cR(G)$ contains no nonzero nilpotent elements, or equivalently if $\cI(G)=\sqrt{\cI(G)}$ is a radical ideal.

For a manifold $M$, we set
$R(M) =R(\pi_1(M))$ and $X(M) = X(\pi_1(M))$ for the spaces of representations and character variety; and $\cR(M) = \cR(\pi_1(M))$ and $\fX(M) = \fX(\pi_1(M))$ for the universal character ring and character scheme.
We will be mainly interested in the case when $M$ is a compact 3-manifold with boundary a torus; typically $M$ will be the complement $\Si \sm \tau(K)$ of a knot $K$ in an integral homology 3-sphere $\Si$. In any case, it is well known that every component $X_j$ of $X(M)$ has $\dim X_j \geq 1,$ see \cite[Proposition 2.4]{CCGLS}. In the case $X_j$ is a curve, there is a smooth projective curve $\wt{X}_j $ and a birational equivalence $\wt{X}_j \to X_j$, and we refer to points $\hat{x} \in \wt{X}_j$  where $\wt{X}_j \to X_j$ has a pole as \emph{ideal points}. Notice that the set of ideal points is Zariski closed and hence finite.   

\subsection{The SL(2,\,$\CC$) Casson invariant}
We briefly recall the definition of the $\SLC$ Casson invariant.
Suppose $\Si$ is a closed, orientable 3-manifold
with a Heegaard splitting $(W_1, W_2, S)$. Here, $S$
is a closed orientable surface embedded in $\Si$, and $W_1$ and $W_2$ are handlebodies
with boundaries $\partial W_1 = S = \partial W_2$
 such that $\Si = W_1\cup_S W_2$.
The inclusion maps $S \hookrightarrow W_i$ and $W_i \hookrightarrow \Si$
induce surjections of fundamental groups. On the level of character varieties, this identifies $X(\Si)$  as the intersection
$$X(\Si) = X(W_1) \cap X(W_2) \subset X(S).$$

There are natural orientations on all the character varieties
determined by their complex structures. The  invariant
$\la_\SLC(\Si)$ is defined as an oriented intersection number of
$X^*(W_1)$ and $X^*(W_2)$ in $X^*(S)$ which counts only compact,
zero-dimensional components of the intersection. Specifically,
there exist a compact neighborhood $U$ of the zero-dimensional
components of $X^*(W_1)\cap X^*(W_2)$ which is disjoint from the
higher dimensional components of the intersection and an isotopy
$h\colon X^*(S) \to X^*(S)$ supported in $U$ such that
$h(X^*(W_1))$ and $X^*(W_2)$ intersect transversely in $U$.
Given a zero-dimensional component $\{\chi\}$ of
$h(X^*(W_1))\cap X^*(W_2)$, we set $\varep_\chi = \pm 1$,
depending on whether the orientation of $h(X^*(W_1))$  followed by
that of $X^*(W_2)$ agrees with or disagrees with the orientation
of $X^*(S)$ at $\chi$.

\begin{definition} Let
$\la_\SLC(\Si) = \sum_\chi \varep_\chi,$
where the sum is  over all zero-dimensional
components  of the intersection $h(X^*(W_1))\cap X^*(W_2)$.
\end{definition}

\subsection{The surgery formula for small knots} 
In this subsection, we recall from \cite{C01} the  surgery formula for the Casson $\SLC$ invariant for Dehn surgeries on small knots in integral homology 3-spheres. 
 
Given a compact, irreducible, orientable 3-manifold $M$ with
boundary a torus, an {\sl incompressible surface} in $M$ is a
properly embedded surface $(S,\partial S) \hookrightarrow (M,\partial M)$ such
that $\pi_1(S) \lto \pi_1(M)$ is injective and no component of $S$ is
a 2-sphere bounding a 3-ball. The surface $S$ is {\sl essential} if
it is incompressible and has no boundary parallel components. 
A 3-manifold is called {\sl small} if it 
does not contain a closed essential surface, and a knot $K$ in
$\Si$ is called {\sl small} if its complement $\Si \sm
\tau(K)$ is a small manifold.

If $\ga$ is a simple closed curve in $\partial M$, let $M_\ga$ be the Dehn
filling of $M$ along $\ga$; it is the
closed 3-manifold obtained by identifying a solid torus with $M$
along their boundaries so that $\ga$ bounds a disk. Note that the
homeomorphism type of $M_\ga$ depends only on the {\sl slope} of
$\ga$ -- that is, the unoriented isotopy class of $\ga$. Primitive
elements in $H_1(\partial M; \ZZ)$ determine slopes
under a two-to-one correspondence.

If $S$ is an essential surface in $M$
with nonempty boundary, then all of its  boundary
components are parallel and the slope of one (and hence all) of these
curves is called the {\sl boundary slope} of $S$. A
slope is called a {\sl strict boundary slope} if it is
the boundary slope of an essential surface that is not the
fiber of any fibration of $M$ over $S^1$.

For $\ga \in \pi_1(M),$ there is a regular map $I_\ga\colon X(M)
\lto \CC$ defined by $I_\ga(\chi) = \chi(\ga).$ Let
$e\colon H_1(\partial M;\ZZ)\lto \pi_1(\partial M)$ be the
inverse of the Hurewicz
isomorphism.  Identifying $e(\xi) \in \pi_1(\partial M)$ with its
image in $\pi_1(M)$ under $\pi_1(\partial M) \lto
\pi_1(M),$ we obtain a well-defined function $I_{e(\xi)}$ on $X(M)$
for $\xi \in H_1(\partial M; \ZZ).$ Let $f_\xi\colon X(M) \lto \CC$
be the regular function defined by $f_\xi  = I_{e(\xi)} -2$ for
$\xi \in H_1(\partial M; \ZZ).$

For any algebraic
component $X_j$ of $X(M)$ with $\dim X_j =1$,
let $f_{j,\xi}\colon X_j \lto \CC$ be the regular
function obtained by restricting $f_{\xi}$ to $X_j$.
Let $\wt{X}_j$ denote the smooth, projective curve
birationally equivalent to $X_j$. Regular functions on $X_j$
extend to rational functions on $\wt{X}_j$, and we abuse
notation and use $f_{j,\xi}$ also for the extension  $f_{j, \xi}\colon \wt{X}_j \lto \CC
\cup \{ \infty \} = \CC \PP^1.$

 \begin{definition} \label{CS-seminorm}
 Let $r\colon X(M)\lto X(\partial M)$ be the restriction map.
Given a one-dimensional component $X_j$ of
$X(M)$ containing an irreducible character such that $r(X_j)$ is also one-dimensional, 
define
the seminorm  $\| \cdot \|_{j}$ on $H_1(\partial M; \RR)$   by setting
$$ \| \xi \|_j =  \deg(f_{j,\xi})$$
for all $\xi \in H_1(\partial M; \ZZ)$.
We refer to $\| \cdot \|_{j}$ as the {\em Culler--Shalen semi-norm}
associated to $X_j$, and we say $X_j$ is   a {\sl norm curve}
if $\| \cdot \|_j$ defines a norm on $H_1(\partial M; \RR)$.
\end{definition}

Note that if $M$ is hyperbolic, then any algebraic component $X_0$ of $X(M)$ containing the character $\chi_{\rho_0}$ of a discrete faithful irreducible representation $\rho_0 \colon \pi_1(M) \to \SLC$ is a norm curve, see \cite[Section 1.4]{CGLS}.

If $M$ is the complement of a small knot $K$, the $\SLC$ Casson invariant of a Dehn filling is closely related to this semi-norm; however we must impose certain
restrictions on the  slope of the Dehn filling.

\begin{definition} \label{irregular}
The slope of a simple closed curve $\ga$ in $\partial M$
is called {\sl irregular} if there exists an irreducible representation
$\rho\colon \pi_1(M) \lto \SLC$ such that
\begin{itemize}
 \item [(i)] the character $\chi_{\rho}$ of $\rho$ lies on a one-dimensional
component $X_j$ of $X(M)$ such that $r(X_j)$ is one-dimensional,
 \item [(ii)] $\tr \rho (\al) =\pm 2$ for all
 $\al$ in the image of $i^*\colon \pi_1(\partial M) \lto \pi_1(M),$
 \item[(iii)] $\ker (\rho \circ i^*)$ is the cyclic group generated by $[\ga] \in \pi_1(\partial M)$.
\end{itemize}
A slope is called {\sl regular} if it is not irregular.
\end{definition}

If $M$ is the complement of a knot $K$ in an integral homology
3-sphere $\Si$, then the meridian $\mer$ and longitude $\lng$ of the knot $K$
provide a preferred basis for $H_1(\partial M; \ZZ).$ We say that the curve $\ga = p \mer + q \lng$ has slope $p/q$ and denote by $M_{p/q} = M_\ga$ the 3-manifold obtained by  $p/q$--Dehn surgery along $K$.

\begin{definition}
A slope $p/q$ is called {\sl admissible} for $K$ if
\begin{enumerate}
\item[(i)] $p/q$ is a regular slope which is not a strict boundary slope, and
\item[(ii)] no $p'$-th root of unity is a root of the Alexander polynomial of
$K$, where $p'=p$ if $p$ is odd and $p' = p/2$ if $p$ is even.
\end{enumerate}
\end{definition}

The next result is a restatement of Theorem 4.8 of \cite{C01}, as
corrected in \cite{C03}.
\begin{theorem} \label{surg-form}
Suppose $K$ is a small knot in an integral homology 3-sphere $\Si$ with
complement $M$. Let $\{X_j\}$ be the collection of all
one-dimensional components of the character variety $X(M)$ such
that $r(X_j)$ is one-dimensional and such that $X_j \cap X^*(M)$
is nonempty. Define $\si\colon \ZZ \lto \{0,1\}$ by $\si(p) \equiv p \mod 2$. 

Then there exist integral weights $m_j > 0$ depending only on
$X_j$ and non-negative numbers $E_0, E_1 \in \frac{1}{2}
\ZZ$ depending only on $K$ such that for every admissible slope
$p/q$, we have
$$
\la_\SLC (M_{p/q}) = \frac{1}{2} \sum_j m_j \| p\mer + q \lng \|_{j} - E_{\si(p)}.
$$
\end{theorem}

We briefly recall some useful properties of the $\SLC$ Casson invariant and
 we refer to \cite{C01} and \cite{BC06} for further details.

On closed 3-manifolds $\Si,$ the invariant $\la_\SLC(\Si) \geq 0$ is nonnegative, satisfies $\la_\SLC(-\Si) = \la_\SLC(\Si)$ under
orientation reversal, and is additive  under connected sum  of $\ZZ/2$--homology 3-spheres (cf.~Theorem 3.1, \cite{BC06}). If $\Si$ is hyperbolic, then $\la(\Si) > 0$
by Proposition 3.2 of \cite{C01}.

If $K$ is a small knot
in an integral homology 3-sphere $\Si$, then
Theorem \ref{surg-form} implies that  the difference
$\frac{1}{p}(\la_\SLC (M_{p/(p+q)}) - \la_\SLC (M_{p/q}))$
is independent of $p$ and $q$ provided $p$ and $q$ are relatively prime and $q$ is chosen sufficiently large.
This allows one to define an  invariant of small knots $K$ in homology 3-spheres by setting
\begin{equation}\label{knotinv}
\la'_\SLC(K) = \la_\SLC (M_{1/(q+1)}) - \la_\SLC (M_{1/q})
\end{equation}
for $q$ sufficiently large.

\subsection{The $\wh{A}$--polynomial}

We begin with the definition of the $A$-polynomial $A_K(m,\ell)$  from \cite{CCGLS} (see also \cite{CL96,CL97}). Given a knot
$K$ in a homology 3-sphere $\Si$, let $M = \Si \sm \tau(K)$ be its complement and choose  a standard meridian-longitude pair $(\mer,\lng)$ for $\pi_1(\partial M)$. Set
$$\La = \{ \rho\colon\pi_1(\partial M) \lto \SLC \mid \text{$\rho(\mer)$ and  $\rho(\lng)$ are diagonal matrices} \}$$
and define the eigenvalue map  $\La \lto \CC^* \times \CC^*$ by setting $\rho \mapsto (u,v) \in \CC^* \times \CC^*$, where $$\rho(\mer) = \begin{pmatrix}u& 0 \\ 0 &u^{-1}\end{pmatrix} \text{ and }\rho(\lng) = \begin{pmatrix}v& 0 \\ 0 &v^{-1}\end{pmatrix}.$$
This map identifies $\La$ with  $\CC^* \times \CC^*$, and the natural projection $t\colon \La \lto X(\partial M)$ is a degree 2, surjective, regular map.

The natural inclusion $\pi_1(\partial M) \lto \pi_1(M)$ induces a map $r\colon X(M) \lto X(\partial M)$, which is regular. We define $V \subset X(\partial M)$ to be the Zariski closure of the union of the images $r(X_j)$ over each component $X_j \subset X(M)$ which contains an irreducible character and for which $r(X_j)$ is one-dimensional, and we set
$D \subset \CC^2 $ to be the Zariski closure of the algebraic curve $t^{-1}(V) \subset \La,$ where we  identify $\La$ and  $\CC^* \times \CC^*$ via the eigenvalue map.
The $A$-polynomial $A_K(m,\ell)$ is just the defining polynomial of $D \subset \CC^2$;
it is well-defined up to sign by requiring it to have integer coefficients with greatest common divisor one and to have no repeated factors. 
Some authors include the factor $\ell-1$ of reducible characters in $A_K(m,\ell)$, but our convention is to only include components $X_j\subset X(M)$ which contain irreducible characters. Thus $\ell-1$ is a factor of $A_K(m,\ell)$ if and only if there is a component $X_j \subset X(M)$ containing an irreducible character whose restriction $r(X_j) \subset X(\partial M)$ is the curve $\ell - 1$. 

In  \cite{BZ01}, Boyer and Zhang define an $A$-polynomial $\wh{A}_{K, X_j}(m,\ell)$ for each one-dimensional component $X_j$ of $X(M)$ for which $r(X_j)$ is one-dimensional.  (Although Boyer and Zhang assume $X_j$ is a norm curve in this definition, the approach works for any one-dimensional component $X_j$ of $X(M)$.)
Their definition takes the defining polynomial $\wh{A}$ to have factors with multiplicities given by the degree of the restriction map $r|_{X_j}$ rather than requiring that the polynomial have no repeated factors.
 Taking the product
$$\wh{A}_K(m,\ell) =  \wh{A}_{K, X_1}(m,\ell) \cdots \wh{A}_{K, X_n}(m,\ell)$$ over one-dimensional components $X_j$ of $X(M)$ different from the component of reducibles, and this gives an alternative version of the $A$-polynomial  that includes factors with multiplicities. Note that only one-dimensional components $X_j$ of $X(M)$ with one-dimensional image $r(X_j)$ contribute to $\wh{A}_K(m,\ell).$ 
As with the $A$-polynomial, by convention the component of reducible characters does not contribute to $\wh{A}_K(m,\ell).$
For small knots, it is not difficult to check that $A_K(m,\ell)$ and $\wh{A}_K(m,\ell) $ have the same factors, only that $\wh{A}_K(m,\ell)$ may include some repeated factors. 

%%%%%%%%%%%%%%%%%%%%%%%%%%%%%%%%%%%%%%%%
\section{Main Results}
In this section, we give a general definition of the $\SLC$ Casson knot invariant for knots and relate it to the $m$-degree of the $\wh{A}$-polynomial. We prove product formulas for both invariants under the operation of connected sum, and we use Whitehead doubling to construct examples of knots whose character variety contains only components $X_j$ of dimension $\dim X_j > 1.$ 
These knots are nontrivial but have trivial $ \wh{A}$-polynomial and trivial $\SLC$ Casson knot invariant.

We are particularly interested in knots $K$ in integral homology 3-spheres $\Si$ for which satisfy the following property, where $M =  \Si \sm \tau(K)$ is the complement of $K$ in $\Si$:

\medskip \noindent
$(*)$ \hspace{2.6cm} The character scheme of $\fX(M)$ is reduced.
\medskip

Note that $(*)$ is equivalent to the condition that the universal character ring $\cR(M)$ is reduced,
and in \cite{LT11} the authors explain its relationship to the AJ conjecture.
For instance, in \cite[Conjecture 2]{LT11}, Le and Tran conjecture
that $(*)$ holds for all knots in $S^3$, and they point out that it has been verified in numerous cases, including two-bridge knots \cite{L93}, torus knots \cite{M10}, and many pretzel knots \cite{LT11, T13}. 
On the other hand, in \cite{KM} Kapovich and Millson have proved a kind of Murphy's law for 3-manifold groups which implies that there are 3-manifolds whose character schemes are not reduced; thus $(*)$ does not hold in general. 
 
\subsection{The SL(2,\,$\CC$) Casson invariant for knots}
In this subsection, we extend the $\SLC$ Casson knot invariant to knots in integral homology 3-spheres satisfying $(*)$.

\begin{theorem} \label{thm-main}
For any knot $K$ in an integral homology 3-sphere satisfying $(*)$, the limit ${\displaystyle  \lim_{q \to \infty}} \; \tfrac{1}{q} \la_\SLC(M_{p/q})$ exists, is independent of $p$, and equals $\frac{1}{2} \deg_m \wh{A}_K(m,\ell)$.
\end{theorem}

We then define the $\SLC$ Casson knot invariant by setting 
\begin{equation} \label{gen-def}
\la'_\SLC(K) = {\displaystyle \lim_{q \to \infty}} \; \tfrac{1}{q} \la_\SLC(M_{p/q}).
\end{equation}
Here $p$ is fixed, and the limit is taken over all $q$ relatively prime to $p$.
The theorem implies that this gives a well-defined invariant of knots. As a direct consequence of Theorem 
\ref{thm-main}, we deduce:

\begin{corollary} \label{corollary-casson}
For any knot $K$ in an integral homology 3-sphere satisfying $(*)$, we have $\la'_\SLC(K) = \frac{1}{2} \deg_m \wh{A}_K(m,\ell)$. 
\end{corollary}

The rest of this subsection is devoted to proving Theorem \ref{thm-main}, and we begin with a definition.

For any representation $\varrho\colon \pi_1(M) \to \SLC$ that extends over $p/q$--Dehn surgery, the eigenvalues $m,\ell$ of $\varrho(\mer), \varrho(\lng)$ satisfy $m^p \ell^q = 1.$ So for $p,q$ relatively prime, we define $F_{p/q}$ to be the plane curve  given by $m^p \ell^q - 1$ and call $F_{p/q}$ the \emph{surgery curve}. 
 
\begin{lemma} \label{useful-lemma}
For any slope $p/q$, the surgery curve $F_{p/q}$ is non-singular. If $p/q$ and $p'/q'$ are distinct slopes, then $F_{p/q}$ and $F_{p'/q'}$ are transverse. 
\end{lemma}

\begin{proof}
We will show that every point on $F_{p/q}$ is simple, and from this it will follow that $F_{p/q}$ is non-singular. 

Let $F=F_{p/q}$ be the polynomial $m^p \ell^q - 1$. Any solution to $F=0$ must have $m \neq 0$ and $\ell \neq 0,$ and it will be  a simple point so long as one of the partial derivatives $\del F / \del m$ or $\del F / \del \ell$ is non-zero at that point. But $\del F / \del m = p m^{p-1} \ell^q$ and $\del F / \del \ell = q m^p \ell^{q-1}$ are both non-zero at each point on $F$. Thus every point on $F$ is simple and consequently $F$ is non-singular.

Now suppose $p/q$ and $p'/q'$ are distinct slopes and set $F= F_{p/q}$ and $F' = F_{p'/q'}.$ Suppose $(m_0, \ell_0)$ is common solution to $F=0$ and $F'=0$. The equations of the tangent lines to $F$ and $F'$ at $(m_0, \ell_0)$ are given by: 
\begin{equation}\label{tangent}
\begin{split}
p m_0^{p-1} \ell_0^q (m - m_0) + q m_0^p \ell_0^{q-1} (\ell - \ell_0) &= 0, \\ 
p' m_0^{p'-1} \ell_0^{q'} (m - m_0) + q' m_0^{p'} \ell_0^{q'-1} (\ell - \ell_0) &= 0.
\end{split}
\end{equation}
Since $(m_0, \ell_0)$ lies on both surgery curves, we see that $m_0^{-1} \ell_0^{-1} = m_0^{p-1} \ell_0^{q-1} = m_0^{p'-1} \ell_0^{q'-1}$, and dividing both equations in \eqref{tangent} by this common factor, we obtain the tangent lines
$p \ell_0 (m - m_0) + q m_0 (\ell - \ell_0) = 0$ and $p' \ell_0 (m - m_0) + q' m_0 (\ell - \ell_0) = 0$, which have distinct slopes since $p/q$ and $p'/q'$ are distinct. This shows that  $F$ and $F'$ intersect transversely at $(m_0, \ell_0)$.
\end{proof}

We now give an outline of the proof of Theorem \ref{thm-main}. It is established by comparing the $\SLC$ Casson invariant $\la_\SLC(M_{p/q})$ for large $q$ with the algebro-geometric intersection number
$$\sum_{x} I_x (\wh{A}_K \cap F_{p/q})$$ 
of the $\wh{A}$-polynomial with the surgery curve $F_{p/q}$, where the sum is taken over all points in the intersection. 
A critical step in proving the theorem is to show the following: 

\medskip
\noindent{\bf Claim.}
There exists a real number $B$   
such that, for any slope $p/q$ such that $F_{p/q}$ does not divide $\wh{A}_K$, we have 
\begin{equation} \label{full-bound}
\tfrac{1}{2} \sum_x I_x(\wh{A}_K\cap F_{p/q}) - B \leq \la_\SLC(M_{p/q})\leq \tfrac{1}{2} \sum_x I_x(\wh{A}_K\cap F_{p/q}).
\end{equation}

Before proving the claim, we explain how to use it to deduce Theorem \ref{thm-main}.
The following lemma  evaluates $ \lim_{q \to \infty} \;  \tfrac{1}{q} \sum_x I_x (\wh{A}_K \cap F_{p/q}),$ where the limit is taken over all $q$ relatively prime to $p$. Since $\wh{A}_K$ has finitely many factors, the claim excludes finitely many slopes $p/q$. Thus the theorem follows  from the lemma by dividing \eqref{full-bound} by $q$, taking the limit as $q \to \infty$, $q$ relatively prime to $p$, and squeezing.  

\begin{lemma} \label{squeeze}
For any $p$, we have 
${\displaystyle \lim_{q \to \infty} \;  \tfrac{1}{q} \sum_x I_x} (\wh{A}_K \cap F_{p/q}) = \deg_m \wh{A}_K(m,\ell)$.
\end{lemma}

\begin{proof}
If $p$ and $q$ are relatively prime, then we have integers $r,s$ with $pr+qs=1.$ We parameterize solutions to $F_{p/q}(m,\ell)=0$ by setting $m = t^q$ and $\ell = t^{-p}$, where $t \in \CC^*$. Clearly $(m,\ell) = (t^q,t^{-p})$ lies on the curve $F_{p/q}$, and $\ell^{-r} m^s = t^{pr} t^{qs}= t^{pr+qs} = t,$ thus any point on $F_{p/q}$ lies on this parameterization.

Suppose $\deg_m \wh{A}_K(m,\ell) = n.$ Then we can write
$$\wh{A}_K(m,\ell) = m^n \al_n(\ell) + m^{n-1}\al_{n-1}(\ell) + \cdots + m \al_1(\ell) + \al_0(\ell),$$
where each $\al_i(\ell)$ is a polynomial in $\ell.$ Now substitute $m = t^q$ and $\ell = t^{-p}$ to obtain
$$\wh{A}_K(t^q,t^{-p}) = (t^q)^n \al_n(t^{-p}) + (t^q)^{n-1} \al_{n-1}(t^{-p}) + \cdots + t^q \al_1(t^{-p}) + \al_0(t^{-p}).$$
Further, since $t \neq 0$, the roots of this Laurent polynomial are identical to the roots of the polynomial obtained by multiplying by $t^d$, where $d$ is chosen so that $t^d \wh{A}_K(t^q,t^{-p})$ is a polynomial with nonzero constant term. Note that, for large $q$, we can take $d=p\deg \al_0$, which is clearly independent of $q$.

The fundamental theorem of algebra implies that $\sum_x I_x (\wh{A}_K \cap F_{p/q})$ equals the degree of $t^d \wh{A}_K(t^q,t^{-p})$. For $q$ sufficiently large, we have
 $$\deg t^d \wh{A}_K(t^q,t^{-p})=d + nq -p \deg \al_n =nq+p(\deg \al_0 - \deg \al_n).$$
Thus, fixing $p$ and letting $q\to \infty$, we see that the intersection
$\sum_x I_x (\wh{A}_K \cap F_{p/q})$ grows linearly in $q$ with leading coefficient $n = \deg_m \wh{A}_K(m,\ell).$ This completes the proof of the lemma.
\end{proof}

\noindent 
\emph{Proof of Claim.} 
In order to establish the bound \eqref{full-bound}, we compare the $\SLC$ Casson invariant $\la_\SLC(M_{p/q})$ with the intersection number $\tfrac{1}{2} \sum_x I_x (\wh{A}_K \cap F_{p/q})$.  We will show that most points in $\CC^* \times \CC^*$ contribute equally to $\la_\SLC(M_{p/q})$ and $\tfrac{1}{2} \sum_x I_x (\wh{A}_K \cap F_{p/q})$, with the sole exceptions being points of the following types:
\begin{enumerate}
\item[(1)]  Solutions $x=(m,\ell)$ to  both $\wh{A}_K(m,\ell) =0$ and $F_{p/q}(m,\ell)=0$ with $m,\ell =\pm 1$.
\item[(2)] Solutions $x=(m,\ell)$ to both $\wh{A}_K(m,\ell) =0$ and $F_{p/q}(m,\ell)=0$ with $\ell = 1$ and $m^2$ equal to a root of the Alexander polynomial $\Delta_K(t)$.
\item[(3)] Solutions $x=(m,\ell)$ to both $\wh{A}_K(m,\ell) =0$ and $F_{p/q}(m,\ell)=0$ with $t(x) = r(\hat{x})$ for an ideal point $\hat{x} \in \wt{X}_j$.  
\end{enumerate}

Suppose $F_{p/q}$ does not divide $\wh{A}_K$ and that $x \in \wh{A}_K \cap F_{p/q}$ is not a point of type (1), (2), or (3).
Writing $\wh{A}_K =  \wh{A}_{K,X_1} \cdots \wh{A}_{K,X_n}$, the basic properties of intersection numbers from \cite[3.3]{F69} imply that
$I_x(\wh{A}_{K} \cap F_{p/q}) = \sum_j I_x(\wh{A}_{K,X_j} \cap F_{p/q}).$
Under the assumption $(*)$, since the ring $\cR(G)$ is independent of the presentation for $\pi_1(M)$, it is clear that the intersection multiplicity $n_{Y_j}$ given in the proof of Proposition 4.3 of \cite{C01} is equal to one for each component $X_j$ of $X(M)$, and therefore by the same proposition
the contribution of $x$ to $\la_\SLC(M_{p/q})$ is $\sum_j d_{j} i_{X_j,x}$,  where $d_j = \deg(r|_{X_j})$ 
and $i_{X_j,x} = \frac{1}{2} I_x(E_j \cap F_{p/q})$, where $E_j$ is the unique polynomial with no repeated factors and integer coefficients vanishing on $\overline{t^{-1}(r(X_j))}.$ The factor of 1/2 is due to the fact that $t\colon \La  \to X(\partial M)$ is generically two-to-one. Here note that the proof of Proposition 4.3 of \cite{C01} holds for any one-dimensional component $X_j$ of $X(M)$ such that $r(X_j)$ is one-dimensional even if $M$ contains a closed incompressible surface.
Note further that by Corollary 2 on p. 75 of \cite{Sf94}, if $Y$ is a component of $X(M)$ with $\dim Y >1$, then the intersection $Y \cap r^{-1}(t(F_{p/q}))$ does not contain any zero-dimensional components and thus $Y$ does not contribute to the Casson $\SLC$ invariant $\la_\SLC(M_{p/q})$.

Since $\wh{A}_{K,X_j}$ is defined as a curve with multiplicity  
$d_j = \deg(r|_{X_j})$, it follows that
$\wh{A}_{K,X_j} = E_j^{d_j},$ and this implies that
$I_x(\wh{A}_{K,X_j} \cap F_{p/q}) = d_j \; I_x(E_j \cap F_{p/q})$.
This shows that points $x \in \CC^* \times \CC^*$ which are not of types (1)--(3) contribute equally to $\la_\SLC(M_{p/q})$ and $\frac{1}{2} \sum_x I_x (\wh{A}_K \cap F_{p/q})$. 

In contrast, points of types (1)--(3) may
contribute differently to $\la_\SLC(M_{p/q})$ and $\frac{1}{2} \sum_x I_x(\wh{A}_K \cap F_{p/q})$ as follows:
 
Points of type (1) need not correspond to points in the character variety $X(M_{p/q})$ and will therefore sometimes contribute less to 
$\la_\SLC(M_{p/q})$ than to $\frac{1}{2} \sum_x I_x(\wh{A}_K \cap F_{p/q}).$ (For more details, see Section 4.1 of \cite{C01}.)

By \cite[Proposition 6.2]{CCGLS}, we see that points of type (2) correspond to images of reducible characters in $X(M_{p/q})$ and thus will contribute less to $\la_\SLC(M_{p/q})$ than to $\frac{1}{2} \sum_x I_x(\wh{A}_K \cap F_{p/q})$.

Points of type (3) correspond to images of ideal characters and as such will also contribute less to $\la_\SLC(M_{p/q})$ than to $\frac{1}{2} \sum_x I_x (\wh{A}_K \cap F_{p/q})$.

Since, in all three cases, the points $x$ never contribute more to $\la_\SLC(M_{p/q})$ than to $\frac{1}{2} \sum_x I_x (\wh{A}_K \cap F_{p/q})$, it follows that 
\begin{equation}\label{upper-bound}
\la_\SLC(M_{p/q}) \leq \tfrac{1}{2} \sum_x I_x(\wh{A}_K \cap F_{p/q}).
\end{equation}

We will now show that there are at most finitely many points of types (1), (2), and (3). A type (1) point $x$ satisfies $x =(m,\ell) = (\pm 1, \pm 1);$ thus there are at most four of them. Because the Alexander polynomial $\Delta_K(t)$ has finitely many roots, there are finitely many points of type (2). Since the number of ideal points on any one-dimensional component $X_j$ is finite, and since $X(M)$ has finitely many components, it follows that there are finitely many type (3) points.  

We determine a bound $B$ for the sum of the contributions of points of types (1)--(3) to $\frac{1}{2} \sum_x I_x (\wh{A}_K \cap F_{p/q})$ that is independent of $p$ and $q$. Note that this bound is not immediate, as while the number of points of types (1)--(3) is finite, they can nevertheless lie on infinitely many surgery curves.   For example, consider the point $x = (e^{6\pi i/5}, e^{2 \pi i/5})$, which satisfies $m = \ell^3$ and $\ell^5 = 1$. Then $x$ lies on $F_{p/q}$ whenever $p$ is not a multiple of 5. If $x$ were a point of type (3), then it could possibly lie on infinitely many surgery curves $F_{p/q}$, for $p$ fixed and $q \to \infty.$ Whether or not points of type (3) exist at all is, to the best of our knowledge, an open question, see \cite[Question 5.1]{CL96} and \cite[Theorem 3.3]{Ch}. (Note that those papers refer to type (3) points as \emph{holes}.)

Now suppose $x$ is a type (1), (2), or (3) point.  Since the surgery curves $F_{p/q}$ are all non-singular and pairwise transverse, at most finitely many of them, say $F_{p_1/q_1}, \ldots, F_{p_k/q_k}$, will intersect $\wh{A}_K$ non-transversely at $x$. 
If such non-transverse intersections exist, then for $i=1,\ldots, k,$ let $b_{x,i} =  \frac{1}{2} I_x(\wh{A}_K \cap F_{p_i/q_i})$ and set $B_x=\max \{ b_{x,1}, \ldots, b_{x,k} \}.$ If every curve $F_{p/q}$ which meets $\wh{A}_K$ at $x$ intersects $\wh{A}_K$ transversely at $x$ then set $B_x = \frac{1}{2} I_x(\wh{A}_K \cap F_{p_0/q_0})$, where $F_{p_0/q_0}$ is an arbitrary surgery curve containing $x$. 
Then for any surgery curve $F_{p/q}$, we have $\frac{1}{2} I_x(\wh{A}_K \cap F_{p/q}) \leq B_x.$
Setting $B=\sum_x B_x$, with the sum taken over all type (1), (2), and (3) points, we see that
\begin{equation}\label{B-bound}
\tfrac{1}{2} \sum_x I_x (\wh{A}_K \cap F_{p/q}) - \la_\SLC(M_{p/q})\leq B.
\end{equation}
Combining Equations \eqref{upper-bound} and \eqref{B-bound} gives \eqref{full-bound}, and this completes the proof of the claim and the proof of the theorem.
\qed
 
 We conclude this subsection with two observations. First, note that the definition of the knot invariant $\la'_\SLC(K)$ is an important first step in developing a Dehn surgery formula for the $\SLC$ Casson invariant. However as is clear from the proof above, a complete surgery formula must include correction terms measuring the difference $\la_\SLC(M_{p/q}) - \frac{1}{2} \sum_x I_x (\wh{A}_K\cap F_{p/q})$ at slopes $p/q$ for which the intersection $\wh{A}_K \cap F_{p/q}$ contains points of type (1), (2), and (3). Corrections for points of type (1) may be made analogously to the corrections for small knots in \cite{C01}; in this case the correction terms will depend only on the parity of $p$. For points of type (2) and (3), new arguments will be needed.

Let $G$ be a compact group and suppose that the Casson invariant $\la_G(\Si)$ has been defined for homology 3-spheres. Then the associated knot invariant is defined as the difference $\la'_G(K) = \la_G(M_{1/{(q+1)}}) - \la_G(M_{1/q})$, which one must show is independent of $q$. In that case, the limit ${\lim_{q \to \infty}} \, \tfrac{1}{q} \, \la_G(M_{1/q})$ exists and equals $\la'_G(K)$. In our case, we have seen that the limit \eqref{gen-def} is independent of $p$, and our proof shows that 
 $\la'_\SLC(K) = \frac{1}{p} (\la_\SLC(M_{p/(q+p)}) - \la_\SLC(M_{p/q}))$ provided $q$ and $p$ are relatively prime and $F_{p/(p+q)}$ and $F_{p/q}$ do not contain any points of types (1), (2), or (3). 
 
\subsection{The $\wh{\mathbf A}$-polynomial for connected sums}
In the next result, we present a product formula for the  $\wh{A}$-polynomial under connected sum of two knots, (cf. Proposition 4.3 of \cite{CL96}, where a similar result for the $A$-polynomial is established).

\begin{theorem} \label{A-hat-connect}
If $K_1$ and $K_2$ are two  oriented knots in $S^3$, then
 $$\wh{A}_{K_1 \# K_2}(m,\ell) = \wh{A}_{K_1}(m,\ell) \cdot \wh{A}_{K_2}(m,\ell).$$
\end{theorem}

\begin{proof}
Let $M_1 = S^3 \sm \tau(K_1)$ and $M_2 = S^3 \sm \tau(K_2)$ be the complements of $K_1$ and $K_2,$ respectively.
Further, let $K=K_1 \#K_2$ be the connected sum of the two knots and $M =  S^3 \sm \tau(K)$ be the complement.
Then by Seifert-van Kampen, for appropriately chosen meridians  $\mer_1$ and $\mer_2$ for $K_1$ and $K_2$, we see that
$$\pi_1(M) = \pi_1(M_1) *_\varphi \pi_1(M_2)$$ is an amalgamated product under the homomorphism $\varphi\colon \langle \mer_1\rangle \to \langle \mer_2\rangle$ given by $\varphi(\mer_1) = \mer_2.$ It follows that the representation space $R(M)$ can be viewed as a subset of the product $R(M_1) \times R(M_2)$, namely
$$R(M) =\{ (\rho_1, \rho_2 ) \in R(M_1) \times R(M_2) \mid \rho_1 (\mer_1)=\rho_2(\mer_2) \}.$$
Given $\rho_i \in R(M_i)$ for $i=1,2$ such that $\rho_1 (\mer_1)=\rho_2(\mer_2),$ we denote
the associated point in $R(M)$ by $\rho_1 *\rho_2$. Let $\mer$ be the meridian of $K_1\# K_2$ corresponding to $\mer_1$ and $\mer_2$ under this identification.

Given a representation $\rho_1\colon \pi_1(M_1) \to \SLC$ such that $\rho_1(\mer_1)$ is conjugate to a diagonal matrix, we can pull it back along the surjection $\pi_1(M) \to \pi_1(M_1)$ to get  a representation $\rho \colon \pi_1(M) \to \SLC$, and in that case
$\rho =  \rho_1 *\rho_2$, where $\rho_2\colon \pi_1(M_2) \to \SLC$ is abelian. 
The meridians and longitudes of $K_1,K_2,$ and $K=K_1\#K_2$ are related by $\mer = \mer_1 = \mer_2$ and 
$\lng = \lng_1\lng_2$, and since $\rho_2$ is abelian, we see that $\rho_2(\lng_2)=I.$ It follows
that $\rho(\mer) = \rho_1(\mer_1)$ and $\rho(\lng) = \rho_1(\lng_1) \rho_2(\lng_2) = \rho_1(\lng_1).$

Note that at most finitely many characters in $X(\partial M_1)$  are characters of representations taking the meridian to a matrix of trace $\pm 2$, and any representation $\rho_1\colon \pi_1(M_1) \to \SLC$ which does not take the meridian to a matrix of trace $\pm 2$ can be conjugated so that $\rho_1(\mer_1)$ is diagonal.
Thus, using the correspondence from the previous paragraph, for any one-dimensional component $X'_j$ of $X(M_1)$, there is a corresponding one-dimensional component $X_j$ of $X(M)$
such that $r'(X'_j) = r(X_j)$. (Here,  $r'\colon X(M_1) \to X(\partial M_1)$ and $r\colon X(M)\to X(\partial M)$ denote the two restriction maps.)
This implies that $\wh{A}_{K_1, X'_j}(m,\ell)$ and $\wh{A}_{K, X_j}(m,\ell)$ 
contain the same factors. 

We now argue that the multiplicities $d'_{j}$ and $d_j$ of the factors in $\wh{A}_{K_1,X_j}(m,\ell)$ and $\wh{A}_{K,X_j}(m,\ell)$ agree. To see this,  recall that the multiplicities are defined in terms of the degree of the restriction maps, which in turn are defined as the cardinality of a generic fiber. 
Choosing $\chi \in r'(X'_{j})$ a generic point  so that $(r')^{-1}(\chi)$ consists entirely of irreducible characters, and noting that the pullback construction gives a one-to-one correspondence between irreducible representations $\rho_1\colon \pi_1(M_1) \to \SLC$ and irreducible representations  $\rho=\rho_1 * \rho_2\colon \pi_1(M) \to \SLC$ with $\rho_2$ abelian, it follows that  
 \begin{eqnarray*}
 d'_{j} &=&\deg \left(r'|_{X'_{j}} \colon X'_{j} \to X(\partial M_1)\right) \\
 &=& \#\left((r')^{-1}(\chi) \cap X'_{j}\right) \\ 
 &=& \#\left(r^{-1}(\chi) \cap X_j\right) \\ 
 &=&\deg\left(r|_{X_j} \colon X_j \to X(\partial M_1)\right) = d_j.
 \end{eqnarray*}
Since the multiplicities agree, we conclude that $\wh{A}_{K_1, X'_{j}}(m,\ell)=\wh{A}_{K, X_j}(m,\ell).$
A similar argument with the roles of $K_1$ and $K_2$ reversed shows that for any one-dimensional component $X'_{j}$ of $X(M_2)$, there is a one-dimensional component $X_j$ of $X(M)$ such that $\wh{A}_{K_2, X'_{j}}(m,\ell)=\wh{A}_{K, X_j}(m,\ell).$ 

We claim that this accounts for all one-dimensional components of $X(M)$.  The previous argument accounts for all characters of representations for which either $\rho_1$ or $\rho_2$ is abelian. Suppose then $X_j$ is a component in the character variety $X(M)$ containing the character $\chi_\rho$ of an irreducible
representation  $\rho=\rho_1 *\rho_2 \colon \pi_1(M) \to \SLC$ such that neither $\rho_1$ nor $\rho_2$ is abelian.
Suppose further that $r(X_j)$ is one-dimensional, since otherwise it would not contribute to $\wh{A}_{K}(m,\ell).$

Note that if both $\rho_1$ and $\rho_2$ are reducible, then by Proposition 6.1 of \cite{CCGLS}, any eigenvalue $\mu$ of $\rho(\mer) = \rho_i(\mer_i)$ satisfies the condition that $\mu^2$ is a root of both $\Delta_{K_1}(t)$ and $\Delta_{K_2}(t),$ the Alexander polynomials of $K_1$ and $K_2.$ Further, $\rho_1(\lng_1) = I =\rho_2(lng_2)$ since both representations are reducible. (This step uses the fact that $\lng_1$ and $\lng_2$ lie in the second commutator subgroup of $\pi_1(M_1)$ and $\pi_2(M_2)$, respectively.) It follows that $\rho(\lng) = \rho_1(\lng_1)\rho_2(\lng_2) = I$. Thus, under restriction, such representations account for at most finitely many points in $r(X_j)$. Hence without loss of generality we may assume that $\rho_1$ is irreducible. 
 
 Since the meridian $\mer$ normally generates $\pi_1(M),$ we see that $\rho(\mer) \neq \pm I.$ It follows that $\Ga_{\rho(\mer)}$,
 the stabilizer subgroup of $\rho(\mer)$, is one-dimensional.  
  We use the technique known as algebraic bending via the action of the group $\Ga_{\rho(\mer)}$ to show that
 $\dim X_j >1$. (See Section 5 of \cite{JM} for a thorough explanation of this technique.)

 Set $\rho_A =\rho_1 * (A \rho_2 A^{-1})$ for $A \in \Ga_{\rho(\mer)}$. Notice that $\rho_A$ is irreducible, and that it is conjugate to $\rho$ if and only if $A=\pm I.$ Allowing $A$ to vary over $\Ga_{\rho(\mer)}$, the family $\rho_A$ of irreducible representations gives rise to a one-dimensional family $\chi_{\rho_A}$ of irreducible characters in $X_j$ such that 
$r(\chi_{\rho_A}) = r(\chi_\rho)$ under the restriction map $r\colon X(M) \to X(\partial M)$. 
 Thus the one-dimensional family of irreducible characters lies in the  fiber $r^{-1}(\chi_\rho)$, and
 since $r(X_j)$ is one-dimensional by assumption, it follows that $\dim X_j >1.$
This completes the proof of the theorem. 
\end{proof}

Combining Corollary \ref{corollary-casson} and Theorem \ref{A-hat-connect},
we see that the Casson $\SLC$ knot invariant is additive under connected sums in $S^3$ provided the knots satisfy the condition $(*)$. 
\begin{corollary}
Let $K_1$ and $K_2$ be knots in $S^3$ such that $K_1$, $K_2$, and $K_1\# K_2$ satisfy $(*)$. Then 
$$\la'_\SLC(K_1 \# K_2) = \la'_\SLC(K_1) +\la'_\SLC(K_2).$$
\end{corollary}

\subsection{Whitehead doubles}
In this subsection, we present a construction of knots $K$ whose character variety has no one-dimensional components other than the component of reducibles. A specific example is provided by the untwisted Whitehead double of the trefoil. Similar computations for the $SU(2)$ character varieities were developed by Eric Klassen in \cite{K93}, and the idea of adapting his approach to the $\SLC$ setting was suggested by Michael Heusener.

Given a knot $J$ in $S^3$, the Whitehead double of $J$ is the knot obtained by gluing one component of the Whitehead link $L$ (shown in Figure \ref{Whitehead}) into the boundary of a tubular neighborhood of $J$. Alternatively, it is the knot whose complement is constructed by gluing the complement of $J$ to the complement $W = S^3 \sm \tau(L)$ of $L$ by a homeomorphism along the common 2-torus. The specific homeomorphism may introduce twists, and the resulting knot is denoted $K_n$ and called the $n$-twisted double of $J$.

\begin{figure}[hb]
\begin{center}
\leavevmode\hbox{}
\includegraphics[scale=1.20]{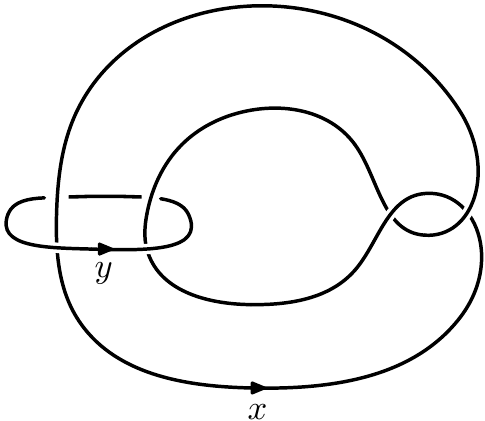}
\caption{The Whitehead link.} 
\label{Whitehead}
\end{center}
\end{figure}

Before providing a detailed construction of $K_n$, we first investigate $W$. 
%Let $L$ be the Whitehead link in $S^3$ in Figure \ref{Whitehead}, and let $W = S^3 \sm \tau(L)$ denote its complement. 
The fundamental group of $W$ admits the presentation (cf. Lemma 9.4 of \cite{CL96}): 
\begin{equation} \label{plug-presentation}
\pi_1(W) = \langle x, y \mid y x y^{-1} x^{-1} y x^{-1} y^{-1} x = x y^{-1} x^{-1} y x^{-1} y^{-1} x y \rangle.
\end{equation}
If $\la_x, \la_y$ denote the longitudes associated to the two components of $L$, then one can further show that 
\begin{equation} \label{longitudes} \begin{split}
\la_x &= y^{-1}xyx^{-1}yxy^{-1}x^{-1}, \\
\la_y &= y^{-1} x^{-1} y x y^{-1} x y x^{-1}.
\end{split} \end{equation}

Writing $L = \ell_1 \cup \ell_2,$ notice that $\ell_1$ and $\ell_2$ are both unknotted, and so the comple-ment $V = S^3 \sm \tau(\ell_2)$ is just the solid torus. 
We will use $y,\la_y$ to also denote the meridian and longitude for $\ell_2$ in $\partial V.$

We now present several lemmas that describe the character variety $X(W)$.
The first lemma identifies the irreducible components of $X(W)$.
%tells us that $X(W)$ consists of two irreducible components $X_1 \cup X_2$, each of dimension two. 
%In the following lemma, we examine the intersection points of these two components and determine the reducible non-abelian characters of $\pi_1(W)$. 
\begin{lemma} \label{two-comps}
The character variety $X(W)$ of representations $\rho \colon \pi_1(W) \to \SLC$ consists of two irreducible components $X_0$ and $X_1$, each of dimension two. The component $X_0$ contains all the reducible characters,
and
the component $X_1$ contains all the irreducible characters. \end{lemma}

\begin{proof}
We begin by describing the component $X_1$ which contains all the irreducible characters.
Since $\pi_1(W)$ admits a presentation with two generators and one relation, any representation
$\rho \colon \pi_1(W) \to \SLC$ which is not parabolic (defined below) can be conjugated so that 
\begin{equation} \label{xy-image}
\rho(x) = \begin{pmatrix} u & s \\ 0 & u^{-1} \end{pmatrix}  \quad \text{ and } \quad \rho(y) = \begin{pmatrix} v & 0 \\ t & v^{-1} \end{pmatrix}
\end{equation}
for $u,v \in \CC^*$ and $s,t \in \CC.$
%This leads to two components of solutions, and we start by describing the first component which contains all the irreducible representations.
Then $\rho$ 
satisfies the relation \eqref{plug-presentation} whenever $u,v,s,t$ satisfy $f(s,t,u,v)=0$, where
\begin{equation}
\begin{split}
f(s,t,u,v)& ~=~ u^2 v^2 (st)^3+ uv(u^2 v^2 - 2 u^2- 2 v^2 + 1)(st)^2  \\
& +(u^4 +v^4 - u^2 (v^2-1)^2- v^2(u^2 - 1)^2) s t  + uv(u^2-1)(v^2-1).
\end{split}
\end{equation}  

This leads to two components of solutions, and we start by describing the component $X_1$ which contains all the irreducible characters. Note that if $\rho$ is irreducible, then either $s\neq 0$ or $t\neq 0$. If $s \neq 0$, then
we can conjugate $\rho$ so that $s=1,$ and since $f(1,t,u,v)$ is irreducible, it follows that the set
$$U_1=\{ (t,u,v) \in \CC \times \CC^* \times \CC^* \mid f(1,t,u,v)=0\}$$
is an irreducible affine variety of dimension two.  Note that for a point $(t,u,v)\in U_1$, the associated character $\chi_\rho$ is reducible if and only if $t = 0$, and it is abelian if and only if $t=0$ and $v=\pm 1.$ 

Switching the roles of $s$ and $t$, we obtain a second affine variety $U_2$. Namely, assuming $t \neq 0$, we can conjugate $\rho$ so that $t=1$, and irreducibility of $f(s,1,u,v)$ implies that  
$$U_2=\{ (s,u,v) \in \CC \times \CC^* \times \CC^* \mid f(s,1,u,v)=0\}$$
is an irreducible affine variety of dimension two. Note again that for $(s,u,v) \in U_2,$ the associated character $\chi_\rho$ is reducible if and only if $s = 0$ and it is abelian if and only if $s=0$ and $u=\pm 1$.

The varieties $U_1$ and $U_2$ provide two affine charts for the first component $X_1$ of $X(W)$, and by construction $X_1$ contains all the irreducible characters. 

To understand the remaining component $X_0$ of $X(W)$ note that any reducible representation $\rho \colon \pi_1(W) \to \SLC$ (including parabolic representations) can be conjugated to be upper triangular. 
However, if
$$\rho(x) = \begin{pmatrix} u & * \\ 0 & u^{-1} \end{pmatrix}  \quad \text{ and } \quad \rho(y) = \begin{pmatrix} v & * \\ 0 & v^{-1} \end{pmatrix},
$$
then the character $\chi_\rho$ is equivalent to the character of an abelian representation $\rho'$
$$\rho'(x) = \begin{pmatrix} u & 0 \\ 0 & u^{-1} \end{pmatrix}  \quad \text{ and } \quad \rho'(y) = \begin{pmatrix} v & 0 \\ 0 & v^{-1} \end{pmatrix},
$$
Thus, every reducible
character $\chi_\rho \in X(W)$ is equivalent to the character of a  diagonal representation,
and any diagonal representation automatically satisfies  \eqref{plug-presentation}.
It follows that the component $X_0$ of reducible characters can be parameterized by $(u,v) \in \CC^* \times \CC^*$, which is clearly an affine variety of  dimension two.
Note further that the reducible characters can be identified with the characters of the representations satisfying \eqref{xy-image} with $s=t=0.$
\end{proof}

In the next lemma, we examine the reducible non-abelian representations of $\pi_1(W)$. 
\begin{lemma} \label{rednonab}
Suppose  $\rho \colon \pi_1(W) \to \SLC$ is a non-abelian reducible representation.
Then there are complex numbers $u,v \neq 0, \pm 1$ such that, up to conjugation,  either
$$\rho(x) = \begin{pmatrix} \pm 1 & 1 \\ 0 & \pm 1 \end{pmatrix}  \quad \text{ and } \quad \rho(y) = \begin{pmatrix} v & 0 \\ 0 & v^{-1} \end{pmatrix}$$
or  
$$\rho(x) = \begin{pmatrix} u & 0 \\ 0 & u^{-1} \end{pmatrix}  \quad \text{ and } \quad \rho(y) = \begin{pmatrix}  \pm 1 & 0 \\ 1 & \pm 1 \end{pmatrix}.$$

For any non-abelian reducible representation $\rho$, its character $\chi_\rho \in X_0 \cap X_1.$
Thus,  non-abelian reducible characters $\chi_\rho$ lie in the closure of the irreducible characters $X^*(W).$
\end{lemma}

\begin{proof}
If $\rho(x)$ and $\rho(y)$ are set equal to either of the pairs of matrices above, one easily checks that \eqref{plug-presentation} is satisfied. We will show these are the only solutions possible for non-abelian reducible representations.

Suppose that $\rho$ is a non-abelian reducible representation, and conjugate it to be upper triangular. Then we have
$$\rho(x) = \begin{pmatrix} u & s \\ 0 & u^{-1} \end{pmatrix}  \quad \text{ and } \quad \rho(y) = \begin{pmatrix} v & t \\ 0 & v^{-1} \end{pmatrix},
$$
where $u,v\in \CC^*$ and $s,t \in \CC$. Note that if $u=\pm 1$ and $v=\pm 1$, then $\rho$ is parabolic and hence abelian. So either $u \neq \pm 1$ or $v \neq \pm 1$.

Assume first of all that $u \neq \pm 1$. Then we can conjugate by upper triangular matrices to arrange that $s=0.$ Since $t \neq 0$ (for otherwise $\rho$ is abelian), we can further conjugate by diagonal matrices and arrange that $t=1$. Then \eqref{plug-presentation} holds if and only if $v=\pm 1.$

If instead $v \neq \pm 1$, then we can conjugate by upper triangular matrices to arrange that $t=0.$ Since $s \neq 0$ (for otherwise $\rho$ is abelian), we can further conjugate by diagonal matrices and arrange that $s=1$. Then \eqref{plug-presentation} holds if and only if $u=\pm 1.$

Obviously, every non-abelian character $\chi_\rho$ lies on $X_0$, and it is equally clear by taking $s=0$ or $t=0$ in \eqref{xy-image}  that $\chi_\rho$ also lies on $X_1$. Further, $X_1$ equals the closure of $X^*(W)$, and that completes the proof of the lemma.
\end{proof}

In the following lemma, we describe the characters $\chi_\rho$ in the intersection $X_0 \cap X_1$ of the two components of $X(W)$ as consisting of reducible non-abelian characters together with the four characters of central representations. 

\begin{lemma} \label{intersect}
Every character in the intersection $X_0 \cap X_1$ is either the character $\chi_\rho$ of a reducible non-abelian representation as in Lemma \ref{rednonab}, or it is the character of a diagonal abelian representation  with $
\rho(x) =\pm I$ or $\rho(y) =\pm I$.
\end{lemma}
 
\begin{proof}
This follows immediately from Lemma \ref{rednonab} and from the description of the abelian representations in $U_1$ and $U_2$ in the proof of Lemma \ref{two-comps}
\end{proof}

We now provide a more detailed construction of the knot $K_n$, the $n$-twisted double of the knot $J$ in $S^3$. Denote the complement of $J$ by $M = S^3 \sm \tau(J).$ Let $\bmu,\bla$ be the meridian and longitude for $J$.
We can specify a homeomorphism $\varphi_n \colon \partial M \to \partial V$ that is unique up to isotopy by requiring $\varphi_n(\bmu) = \la_y$
and $\varphi_n(\bla) = y + n \la_y.$ The image of $\ell_1$ in $M \cup_{\varphi_n} V$ is a knot $K_n$ in $S^3$ that we call the $n$-twisted double of $J$.
Let $Z_n = S^3 \sm \tau(K_n)$ be the complement of the $n$-twisted double, and note that
 $$Z_n = M  \cup_{\varphi_n} W.$$

Using Seifert-van Kampen, we see that $$\pi_1(Z_n) = \pi_1(M) *_{(\varphi_n)_*} \pi_1(W).$$
Given a representation $\rho \colon \pi_1(Z_n) \to \SLC$, by restricting we obtain representations $\rho_1 \colon \pi_1(M) \to \SLC$
and $\rho_2 \colon \pi_1(W) \to \SLC$. Note that the representations $\rho_1, \rho_2$ satisfy
\begin{equation} \label{satisfaction}
\rho_1(\bmu) = \rho_2(\la_y) \quad \text{ and }\quad \rho_1(\bla) = \rho_2\left(y (\la_y)^n\right),
\end{equation}
and that any pair  $(\rho_1,\rho_2) \in R(M) \times R(W)$ of representations satisfying \eqref{satisfaction} uniquely determines a 
representation $\rho \colon \pi_1(Z_n) \to \SLC$. In this case, we write $\rho = \rho_1 * \rho_2$.

\begin{figure}[hb]
\begin{center}
\leavevmode\hbox{}
\includegraphics[scale=1.20]{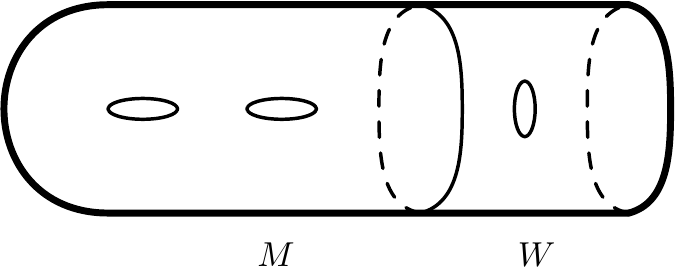}
\caption{The complement $Z_n = S^3 \sm \tau(K_n)$ of the $n$-twisted Whitehead double.} 
\label{twisted-double}
\end{center}
\end{figure}

\begin{lemma} \label{irred-one}
If $\rho \colon \pi_1(Z_n) \to \SLC$ is non-abelian, then its restriction $\rho_2= \rho|_{\pi_1(W)}$ to $W$ is non-abelian. If $\rho_2$ is reducible and non-abelian, then $\chi_{\rho_2}(y) = \pm 2$ and $\chi_{\rho_2}(\la_y) = 2$.
\end{lemma}

\begin{proof}
Suppose $\rho_2$ is abelian.
By \eqref{longitudes} we see that $\la_y$ is a product of commutators,
and it follows that $\rho_2(\la_y)=I$. Thus \eqref{satisfaction} implies that $\rho_1(\bmu) = I$, and since $\bmu$ normally generates $\pi_1(M)$, it follows that $\rho_1$ is trivial. But this implies that $\rho =\rho_1 * \rho_2$ is abelian, a contradiction.

Now suppose $\rho_2$ is reducible and non-abelian. Conjugating, we may assume that $\rho_2$ is upper triangular with
$$\rho_2 (x) = \begin{pmatrix} u & 1 \\ 0 & u^{-1} \end{pmatrix} \quad \text{ and } \quad \rho_2 (y) = \begin{pmatrix} v & t \\ 0 & v^{-1} \end{pmatrix}.$$
Using \eqref{longitudes}, one easily sees that $\rho(\la_y)$ must be upper triangular with trace 2. 
Since $\rho_2(y)$ commutes with $\rho_2(\la_y)$, either $\rho_2(\la_y)=I$ or $v = \pm 1$. 

If $\rho_2(\la_y)=I$, then \eqref{satisfaction} shows that $\rho_1(\bmu) = I$, which implies $\rho_1$ is trivial. Further $\rho_2(y)=\rho_1(\bla)\rho_2(\la_y)^{-n}=I$, which shows that $\rho_2$ is abelian, a contradiction.  Thus $v = \pm 1$, and the lemma follows.
\end{proof}

Recall that $X^*(G)$ denotes the subset of characters $\chi_\rho$ of irreducible representations $\rho \colon G \to \SLC$. For the character variety of the $n$-twisted double $K_n$, we define
\begin{eqnarray*}
X^{*,ab}(Z_n) &=& \{ \chi_\rho \in X^*(Z_n) \mid \rho|_{\pi_1(M)} \text{ is abelian} \} \\
X^{*,na}(Z_n) &=& \{ \chi_\rho \in X^*(Z_n) \mid \rho|_{\pi_1(M)} \text{ is non-abelian} \}. 
\end{eqnarray*}

Let $U_n$ be the $n$-twisted double of the unknot,   
and observe that $U_n$ is a twist knot.

\begin{lemma} \label{nonabel}
$X^{*,ab}(Z_n) \cong X^*(U_n).$
\end{lemma}

\begin{proof}
A representation $\rho_1 \colon \pi_1(M) \to \SLC$ is abelian if and only if it factors through the abelianization map $\pi_1(M) \to H_1(M) \cong \ZZ$.
Thus, irreducible representations  $\rho$ of $\pi_1(Z_n) = \pi_1(M) *_{(\varphi_n)_*} \pi_1(W)$ which are abelian on $\pi_1(M)$ 
are in one-to-one correspondence with representations of  $\ZZ *_{(\varphi_n)_*} \pi_1(W) = \pi_1(S^3 \sm \tau(U_n)).$ 
\end{proof}

The next result gives a slightly stronger statement for untwisted doubled knots $K_0$.
Henceforth for such knots we omit the 0-subscripts, writing $K$ rather than $K_0$ for the untwisted double and $Z$ rather than $Z_0$ for its complement.

\begin{proposition} \label{prop-untwistirr}
Let $J$ be a knot in $S^3$ and $M = S^3 \sm \tau(J)$ its complement. Let $K$ be the untwisted double of $J$ and $Z=S^3 \sm \tau(K)$ its complement.  
\begin{enumerate}
\item[(i)] If $\rho \colon \pi_1(Z) \to \SLC$ is irreducible, then  $\rho_1 = \rho|_{\pi_1(M)}$ is irreducible. 
\item[(ii)] If $\rho \colon \pi_1(Z) \to \SLC$ is reducible, then   
$\rho_1$ is trivial and $\rho$ is abelian.
\end{enumerate}
\end{proposition}
 
\begin{proof} Given a representation $\rho \colon \pi_1(Z) \to \SLC$ with $\rho_1 = \rho|_{\pi_1(M)}$ reducible, we show $\rho_1$ is trivial and $\rho$ is abelian. This will establish both claims of the proposition. 

Suppose $\rho_1$ is reducible. 
Since $\bla$ lies in the second commutator subgroup of $\pi_1(M)$, it follows that $\rho_1(\bla)=I$, and 
\eqref{satisfaction} shows that $\rho_2(y)=I$. Equation \eqref{longitudes} implies $\rho_2(\la_y) =I$, and applying \eqref{satisfaction} again shows that $\rho_1(\bmu)=I$. Since $\bmu$ normally generates $\pi_1(M)$, this implies $\rho_1$ is trivial.  Moreover since $\rho_2(y)=I$ and $\pi_1(W)$ is generated by $x$ and $y$, the image of $\rho=\rho_1 * \rho_2$ is generated by $\rho_2(x)$, and it follows that $\rho$ is abelian.
 \end{proof}

The Alexander polynomial of the $n$-twisted Whitehead double $K_n$ is given by $$\Delta_{K_n}(t)= nt^2+(1-2n)t +n.$$ 
By Proposition 6.1 of \cite{CCGLS}, any reducible non-abelian representation $\rho \colon \pi_1(Z_n) \to \SLC$ has meridional eigenvalue equal to a square root of $\Delta_{K_n}(t)$. Since the untwisted 
double $K$ has trivial Alexander polynomial, this shows that $\pi_1(Z)$ does not admit any non-abelian reducible $\SLC$ representations. In particular, every non-abelian representation  $\rho \colon \pi_1(Z) \to \SLC$ is automatically irreducible, and the component $Y_0 \subset X(Z)$ of reducible characters is  disjoint from the other components $Y_j \subset X(Z).$

The next result shows that for most
untwisted doubles, the variety of irreducible characters has no one-dimensional components.

\begin{proposition} \label{higher-dim}
Let $J$ be a knot in $S^3$ such that the $A$-polynomial is not divisible by $\ell - 1$, $\ell + 1$ or $\ell^2 + m$, and let $M = S^3 \sm \tau(K)$ and $r \colon X(M) \to X(\partial M)$. Suppose that every component $X_j$ of $X(M)$ has one-dimensional image $r(X_j)$. 

If $K$ denotes the untwisted double of $J$ and $Z =S^3 \sm \tau(K)$ its complement, then every component $Y_j$ of $X(Z)$ of irreducible characters has $\dim Y_j > 1$. 
\end{proposition}

Note that most knots satisfy the above hypotheses; for a specific example take $J$ to be the trefoil.

\begin{proof}
We write $Z =M \cup_T W$, where $T$ is the 2-torus along which $M$ and $W$ are identified, and we 
use $r_1$ and $r_2$ to denote the restriction maps from $R(M)$ and $R(W)$ to $R(T)$, giving the diagram:

\begin{center}
\begin{tikzcd} 
R(M) \arrow{dr}{r_1}  && R(W)\arrow{dl}[swap]{r_2}  \\
&R(T) 
\end{tikzcd} 
\end{center}
We use this diagram to identify $R(Z)$ with pairs $(\rho_1, \rho_2) \in R(M) \times R(W)$ which satisfy $r_1(\rho_1) = r_2(\rho_2)$,
and in that case we write $\rho = \rho_1 * \rho_2$. Recall from the presentation \eqref{plug-presentation} of $\pi_1(W)$  and from the gluing equations \eqref{satisfaction} for the untwisted double that 
$$r_1(\rho_1) = r_2(\rho_2) \quad \text{ if and only if } \quad \rho_1(\bmu) = \rho_2(\la_y) \text{ and } \rho_1(\bla)= \rho_2(y).$$

%The character variety $X(W)$ can be seen to consist of two components $X_0(W) \cup X_1(W)$, both with dimension two. The first component $X_0(W)$ consists entirely of reducible characters and the second component $X_1(W)$ contains all the irreducible characters and all non-abelian reducible characters. 
Recall from Lemma \ref{two-comps} that $X(W) = X_0(W) \cup X_1(W)$. Further by Lemma \ref{rednonab},
%In particular, one can show that 
every non-abelian reducible representation $\rho_2 \colon \pi_1(W)\to \SLC$ is a limit of irreducible representations.

In their proof of Lemma 9.4 of  \cite{CL96}, Cooper and Long show that the image of $X^*_1(W)$ under restriction $X(W) \to X(T)$ is
the Zariski-open subset given by the complement of the \emph{forbidden} curves $m-1, m+1, \ell+m^2$. Here $T \subset \partial W$ is the boundary of the component $\ell_2$ of the Whitehead link $L$ labeled $y$ in Figure \ref{Whitehead}, and $m$ is an eigenvalue of $\rho(y)$ and $\ell$ an eigenvalue of $\rho(\la_y).$
In particular any curve of representations of $\pi_1(T)$ with characters in the image of $X_1^*(W)$ extends continuously to a curve of representations of $\pi_1(W)$. 
We use this to show that every component $X_j \subset X(M)$ containing irreducible characters and with image $r(X_j) \subset X(T)$ not coincident with a forbidden curve gives rise to a
component $Y_j \subset X^*(Z)$ of dimension at least 2, and that every component of $X^*(Z)$ arises this way.
 
Let $X_j$ be a component of $X(M)$ containing an irreducible character, and let $R_j$ 
be the corresponding component in the space $R(M)$ of $\SLC$ representations, so that under $t \colon R(M) \to X(M)$, we have $t(R_j) = X_j.$ 
By hypothesis, the image $r(X_j)$ is one-dimensional and does not coincide with any of the forbidden curves $\ell -1, \ell +1,$ and $m+\ell^2$.
(Note that the meridian and longitude of the knot $J$ are opposite to those of $T \subset \partial W$, and thus we must exchange $\ell$ and $m$ when viewing the forbidden curves in $X(\partial M)$.)
Consequently, the intersection of $r(X_j)$ with the three forbidden curves consists of at most finitely many characters.
Set
$$R'_j = \{ \rho_1 \in R_j \mid \text{$\rho_1$ is irreducible and $r(\chi_{\rho_1})$ does not lie on a forbidden curve} \}.$$
Since $X_j$ contains only finitely many reducible characters, under the composition
$R_j \stackrel{t}{\lto} X_j \stackrel{r}\lto X(T)$, this excludes at most finitely many points from $r(X_j)$. Now the image $r(X_j)$ is one-dimensional by hypothesis, and so it follows that $R'_j$ is non-empty and Zariski-open in $R_j$.
Further, Lemma 9.4 of  \cite{CL96} shows that every $\rho_1 \in R'_j$ extends to a representation
$\rho \colon \pi_1(Z) \to \SLC$ whose restriction $\rho_2 = \rho|_{\pi_1(W)}$ is irreducible.

Let $Y_j \subset X(Z)$ be the component of irreducible characters $\chi_\rho$ with $\rho_1 \in R_j$, and let
$f \colon Y_j \to X(T)$ be the map given by $\chi_\rho \mapsto \chi_{\rho_0}$, where $\rho_0 = \rho|_{\pi_1(T)}$ is the restriction of $\rho$ to the splitting torus $T$. By the previous construction, we see that $f(Y_j)$ contains $(r\circ t)(R_j')$, which is a curve with at most finitely many points removed. 
In fact, the construction shows that a Zariski-open subset of $Y_j$ consists of characters $\chi_\rho$ with 
$\rho =\rho_1 * \rho_2$, where $\rho_1 \colon \pi_1(M) \to \SLC$ and $\rho_2 \colon \pi_1(W) \to \SLC$ are both irreducible representations.
By Lemma \ref{irred-one} and Proposition \ref{prop-untwistirr}, every irreducible character $\chi_\rho \in X(Z)$ is the character of a representation $\rho =\rho_1*\rho_2$ with $\rho_1$ irreducible and $\rho_2$ non-abelian. Let $R_j$ be the component of $R(M)$ containing $\rho_1$. If $\rho_2$ is irreducible, then $\chi_\rho \in Y_j$ for the component $Y_j \subset X(Z)$ constructed above. If instead $\rho_2$ is reducible and non-abelian, then it lies in the Zariski closure of $R'_j$,  and it follows that $\rho$ lies in the Zariski closure of $Y_j$.

We now use algebraic bending \cite{JM} to show that $\dim(Y_j)\geq 2.$  
Given $\rho=\rho_1 * \rho_2$ with $\rho_1$ and $\rho_2$ irreducible,  let $\Gamma_0$ be the stabilizer subgroup of the restriction $\rho_0 = \rho|_{\pi_1(T)}.$ Since $\pi_1(T)$ is abelian and $\rho_1$ is irreducible, it follows that $\Gamma_0$ is one-dimensional. (Indeed, $\Gamma_0$ is isomorphic to either $\CC^*$ or $\CC$ depending on whether $\rho_0$ is diagonal or parabolic.)
 For $A \in\Gamma_0$, define $\rho_A = \rho_1* (A\rho_2 A^{-1})$. Clearly then $\rho_A \colon \pi_1(Z) \to \SLC$ is irreducible and $f(\chi_{\rho_A}) = f(\chi_{\rho})$. On the other hand, for $A \neq \pm I,$ one can show that $\rho_A$ is not conjugate to $\rho$. This gives a one-dimensional family of irreducible characters in the fiber $f^{-1}(\chi_{\rho_0})$, and since $f(Y_j)$ is one-dimensional, this implies that $\dim Y_j>1.$  
% To conclude the proof, we note that $X^*(Z)$ contains no additional components. To see this note that if $Y$ is any component of $X^*(Z)$ then the points of $Y$ are characters $\chi_\rho$ for representations $\rho=\rho_1*\rho_2$ such that the characters $\chi_{\rho_1}$ all lie on a component $X_j$ of $X(M)$. The arguments from the proof of Lemma 9.4 of \cite{CL96} show that $\rho_2$ lies on a curve of representations which pair with representations whose characters map out a Zariski open subset of $X_j$, even if $\rho_2$ is nonabelian reducible. It follows that $Y = Y_j$, completing the proof.
% 
\end{proof}
 
\begin{corollary}
If $J$ is a knot in $S^3$ satisfying the hypotheses of Proposition \ref{higher-dim}, then its untwisted  double $K$ has $\wh{A}_{K}(m,\ell) = 1$ and $\la'_{\SLC}(K) = 0$.
\end{corollary}

\begin{proof}
Since $\wh{A}_K(m,\ell)$ is defined using only components of $X(Z)$ of dimension one, we see immediately that $\wh{A}_{K}(m,\ell) = 1$. The assertion $\la'_{\SLC}(K) = 0$ follows from Corollary \ref{corollary-casson} if $K$ satisfies condition $(*)$. If on the other hand $(*)$ does not hold, then no component $Y_j$ of $X(Z)$ with dimension greater than one contributes to $\la_\SLC(Z_{p/q})$ for any surgery $p/q$, as noted in the proof of Theorem \ref{thm-main}. Thus clearly $\la'_{\SLC}(K) = 0$ in this case too.
\end{proof}

Taking $J$ to be the trefoil, this implies that its untwisted double has trivial $\wh{A}$-polynomial and vanishing $\SLC$ Casson knot invariant. In particular, neither the $\wh{A}$-polynomial nor the $\SLC$ Casson knot invariant detect the unknot.

\bigskip
In conclusion, it would be interesting to find a way to incorporate higher-dimensional components of the $\SLC$ character variety $X(\Si)$ into the definition of the $\SLC$ Casson invariant for 3-manifolds $\Si$. The resulting invariant would coincide with $\la_\SLC(\Si)$ in the case the character variety $X(\Si)$ is zero-dimensional, and an intriguing problem would then be to establish a surgery formula for the new invariant and to define an associated invariant of knots. It is reasonable to believe that the knot invariant would be related to an appropriately defined generalization of the 
$A$-polynomial in much the same way that $\la'_\SLC$ is related to the $\wh{A}$-polynomial. In particular, since the $m$-degree of the $A$-polynomial is known to detect the unknot \cite{B12}, one would expect that an $\SLC$ knot invariant that incorporates higher-dimensional components of $X(M)$ would be a powerful tool in low-dimensional topology.

\bigskip
\noindent
{\it Acknowledgements.} Both authors would like to thank Steve Boyer, Marc Culler, Michael Heusener, Thang Le, Darren Long, and Steffen Marcus for generously sharing their expertise.  We would also like to thank the anonymous referee for pointing out an error in an earlier version of this paper. H. Boden gratefully acknowledges the support of a Discovery Grant from the Natural Sciences and Engineering Research Council of Canada.

\end{document}